\documentclass[12pt]{article}

\usepackage{amsthm,amsfonts, amsbsy, amssymb,amsmath,graphicx}
\usepackage{graphics}
\usepackage[english]{babel}
\usepackage{graphicx}
\usepackage{lineno}
\usepackage{enumerate}
\usepackage{tikz}
\usepackage{tkz-graph}
\usepackage{hyperref}
\usepackage{color}
\usepackage{tikz}
\usepackage{algorithm}
\usetikzlibrary{arrows.meta,shapes.arrows}

\hypersetup{colorlinks=true, linkcolor=blue, citecolor=red,urlcolor=blue}

\newtheorem{theorem}{Theorem}

\newtheorem{proposition}[theorem]{Proposition}
\newtheorem{corollary}[theorem]{Corollary}

\newtheorem{observation}[theorem]{Observation}

\title{Efficient closed domination in digraph products\thanks{First author was partially supported by ARRS Slovenia under the grants 
P1-0297 and J1-9109.}}

\author{Iztok Peterin$^{(1)}$ and Ismael G. Yero$^{(2)}$\\
\\
$^{(1)}${\small Faculty of Electrical Engineering and Computer Science}\\
{\small University of Maribor,} {\small Koro\v{s}ka cesta 46, 2000 Maribor, Slovenia.} \\
{\small\it iztok.peterin\@@um.si} \\
$^{(2)}${\small Departamento de Matem\'aticas, Escuela Polit\'ecnica Superior de Algeciras}\\
{\small Universidad de C\'adiz,} {\small Av. Ram\'on Puyol s/n, 11202 Algeciras, Spain.} \\
{\small\it ismael.gonzalez\@@uca.es}
}

\date{}

\begin{document}

\maketitle

\begin{abstract}
\noindent A digraph $D$ is an efficient closed domination digraph if there exists a subset $S$ of $V(D)$ for which the closed out-neighborhoods centered in vertices of $S$ form a partition of $V(D)$. In this work we deal with efficient closed domination digraphs among several product of digraphs. We completely describe the efficient closed domination digraphs among lexicographic and strong products of digraphs. We characterize those direct products of digraphs that are efficient closed domination digraphs, where factors are either two cycles or two paths. Among Cartesian product of digraphs, we describe all such efficient closed domination digraphs such that they are a Cartesian product digraph either with a cycle or with a star.
\medskip

\medskip


\end{abstract}

\section{Introduction}

Partitions of objects in mathematics is a natural procedure and if the elements of the partition satisfy some properties, then it frequently becomes of high interest and usefulness in the research community. Every partition is connected with an equivalence relation and enables a factor structure, which often bring simplification and deeper insight into the properties of the starting object itself. Therefore, it is natural to study different kind of partitions and the existence of them. Number theory and pure combinatorics, for instance, offer plenty of examples where this attracts the attention.

In graph theory, there is a wide range of possibilities and probably one of the most natural ones concerns vertex (open and closed) neighborhoods. The problem of existence of a partition of a vertex set of a graph into closed neighborhoods was relatively early started by Biggs in 1976 (see \cite{biggs-1973}). The theme became quite popular, and through the years, several combinatorial and computational results were presented. For instance, such partition has been studied for bipartite and chordal graphs in~\cite{smart-1995}, where the problem of determining if a graph has such a partition was proved to be NP-complete. In recent years, the direction of researches has been centered into finding which graphs, belonging to some certain graph classes, are such. For only two examples of such type of researches, we suggest \cite{knor-2012} for cubic vertex-transitive graphs and \cite{jha-2014} for twisted tori. Along all such researches, graphs having such a vertex partition were called \emph{1-perfect graphs} (which comes from \cite{biggs-1973}, although such name appeared after), or \emph{efficient domination graphs} (which comes from \cite{bange}). More recently, efficient domination graphs, has been called \emph{efficient closed domination graphs}, in order to be more consequent with the closed neighborhoods which are used (see \cite{sandi-pet-ye}).

On the other hand, the study of graphs that admit a partition of the vertex set into open neighborhoods was initiated later in 1993 by Cockayne \emph{et al.} in \cite{CoHaHe}, where such partition were investigated under the name of total perfect codes. The terminology of efficient open domination graphs was used for the first time by Gavlas and Schultz in 2002 (see \cite{GaSc}). Cayley graphs were studied with respect to this property in~\cite{Tham}, grid graphs were investigated in \cite{CoHeKe, Dej, KlGo} and direct product graphs with such a partition were characterized in~\cite{AbHamTay}. For several other graph products see~\cite{KuPeYe1}, and in particular, for the Cartesian product see~\cite{KPRT}.

For the case of digraphs one needs to consider that, beside open and closed neighborhoods, in addition one would need to separate them into ``out-neighborhoods'' and ``in-neighborhoods'' (in connection with the direction of arrows). However, this additional separation would be ``somehow artificial'' as the one is closely related with the other in the following sense. If we reverse the orientation of every arc, then out-neighborhoods became in-neighborhoods and vice versa. Therefore, we can stay by efficient closed and efficient open domination in digraphs.


Efficient closed domination in digraphs has been introduced in \cite{barkauskas} (according to \cite{schwenk}), and also studied in some works like for instance \cite{huang,martinez,niepel}. On the other hand, investigations on efficient open domination digraphs has been initiated in \cite{schaudt} under the name of efficient total domination digraphs. As in the case of graphs, the literature concerning efficient closed domination digraphs is larger than that on efficient open domination digraphs. Some other results on efficient open domination digraphs can be found in \cite{sohn}.

The theory of efficient closed domination in digraphs (partitions with respect to closed out-neighborhoods) has been moreover studied in other ``not exactly directly'' style. That is, via efficient absorbant digraphs (partitions with respect to closed in-neighborhoods). As mentioned, these two concepts are interlaced and can be considered as one problem. All published works on efficient absorbants are connected with de Bruijn digraphs. The question about characterization of efficient absorbants among De Bruijn's digraphs was posed by Wang \emph{et al.} in \cite{WaWu}. The sufficient condition for it was shown by Wu \emph{et al.} in \cite{WuWa}, with the efficient absorbants. The characterization was then confirmed by Shiau \emph{et al.} in \cite{ShSh}.

Several graph products have been investigated in the last few decades and a
rich theory involving the structure and recognition of classes of these
graphs has emerged, cf. \cite{ImKl}. The most studied graph products are the
Cartesian product, the strong product, the direct product and the
lexicographic product which are also called \emph{standard products}. One
standard approach to graph products is to deduce properties of a product
with respect to (usually the same) properties of its factors. See a short collection
of these type involving efficient closed or open domination in
\cite{abay-2009,DoGrSp,Grav,JeKlSp,KlSpZe,KPRT,KuPeYe1,mollard-2011,taylor-2009,Zero}.

In the next section we fix the terminology. Following section deals with some partial results on efficient closed domination of Cartesian products of digraphs. We fix one factor and explain the structure of the other (similar approach was done for efficient open domination of Cartesian products of graphs in \cite{KPRT}). We continue with a section on the efficient closed domination of direct products of some families. In particular we treat cycles, which posed a problem among graphs, see \cite{JeKlSp,KlSpZe,Zero}. We finish with efficient close dominated digraphs among strong product and lexicographic product of digraphs, where a characterizations of all such products is presented with respect to some properties of factors.

\section{Preliminaries}


Let $D$ be a digraph with vertex set $V(D)$ and arc set $A(D)$.
For any two vertices $u,v\in V(D)$, we write $(u,v)$ as the \emph{arc} with direction or orientation from $u$ to $v$, and say $u$ is \emph{adjacent to} $v$, or $v$ is \emph{adjacent from} $u$. For an arc $(u,v)$ we also say that $u$ is the \emph{in-neighbor} of $v$ and that $v$ is the \emph{out-neighbor} of $u$. For a vertex $v\in V(D)$, the {\em open out-neighborhood} of $v$ ({\em open in-neighborhood} of $v$) is $N_D^{+}(v)=\{u\in S:(v,u)\in A(D)\}$ ($N_D^{-}(v)=\{u\in V(D):(u,v)\in A(D)\}$). The \emph{in-degree} of $v$ is $\delta_D^-(v)=|N_D^{-}(v)|$, the \emph{out-degree} of $v$ is $\delta_D^+(v)=|N_D^{+}(v)|$ and the degree of $v$ is $\delta_D(v)=\delta_D^-(v)+\delta_D^+(v)$. Moreover, $N_D^{-}[v]=N_D^{-}(v)\cup\{v\}$ is the {\em closed in-neighborhood} of $v$ ($N_D^{+}[v]=N_D^{+}(v)\cup\{v\}$ is the {\em closed out-neighborhood} of $v$). In above notation we omit $D$ if there is no ambiguity with respect to the digraph $D$. We similarly proceed with any other notation which uses such style of subscripts.

A vertex $v$ of $D$ with $\delta^+(v)=|V(D)|-1$ is called an \emph{out-universal vertex}, and if $\delta^-(v)=|V(D)|-1$, then $v$ is called an \emph{in-universal vertex}. A vertex $v$ of $D$ with $\delta^+(v)=0$ is called a \emph{sink}, and if $\delta^-(v)=0$, then $v$ is called a \emph{source}. If $\delta(v)=0$, then $v$ is an \emph{isolated vertex} or \emph{singleton}. An arc of the form $(v,v)$ is called a \emph{loop} and can be considered as a directed cycle of length one. A vertex $v$ with $\delta(v)=1$ is called a \emph{leaf} and is either a sink (if $\delta^+(v)=0$) or a source (if $\delta^-(v)=0$). Clearly, any vertex $u$ with $\delta(u)=2$ is either a sink, or a source, or $\delta^-(u)=1=\delta^+(u)$.
In a directed cycle $C_k$ all vertices are of degree two. It is easy to see that the number of sinks is the same as the number of sources in every directed cycle. Therefore we will denote a directed cycle on $k$ vertices with $p$ sources (and with $p$ sinks) as $C_k^p$.

The \emph{underlying graph} of a digraph $D$ is a graph $G_D$ with $V(G_D)=V(D)$ and for every arc $(u,v)$ from $D$ we have an edge $uv$ in $E(G_D)$. For those cases in which double arcs exist, only one edge is considered in the underlying graph. For a digraph $D$, we also define the digraph $D^-$ on the same set of vertices where $(u,v)\in A(D^-)$ whenever $(v,u)\in A(D)$. In other words, if we change the orientation of every arc in $D$, then we obtain $D^-$.

We use the notation $D[S]$ for the subdigraph of a digraph $D$ induced by the vertices of $S\subseteq V(D)$. Different kinds of connectivity are known for digraphs. Here we are interested in only one. That is, we say that a digraph $D$ is \emph{connected} if its underlying graph $G_D$ is connected. Clearly, the \emph{components} of $D$ are then the same as components of $G_D$.

\subsection{Domination in digraphs}

Let $D$ be a digraph and let $S\subseteq V(D)$. The set $S$ is called a \emph{dominating set} of $D$ if the closed out-neighborhoods centered in vertices of $S$ cover $V(D)$, that is $V(D)=\bigcup_{v\in S}N^+_D[v]$. Let $Q,R\subset V(D)$. If the vertices of $Q$ cover $R$, that is $R\subseteq \bigcup_{v\in Q}N^+_D[v]$, then we will use the symbol $Q\rightarrow R$. On the contrary, if every vertex from $R$ is not adjacent from any vertex of $Q$, then we write $Q\nrightarrow R$ or even $Q\nleftrightarrow R$, if there are no arcs between vertices from $R$ and $Q$ in any direction.

Let $S$ be a dominating set of $D$. If $N^+_D[v]\cap N^+_D[u]=\emptyset$ for every two different vertices $u,v\in S$, then the set $\{N^+_D[v]: v\in S\}$ not only cover $V(D)$ but also \emph{partition} $V(D)$. In this case we say that $S$ is an \emph{efficient closed dominating set} of $D$ (or ECD set for short) of $D$. If there exists an ECD set $S$ for the digraph $D$, then $D$ is called an \emph{efficient closed domination digraph} (or ECD digraph for short).

The minimum cardinality of a dominating set of $D$ is called the \emph{domination number} of $D$ and is denoted by $\gamma (D)$. If $S$ is a dominating set of cardinality $\gamma (D)$, then we say that $S$ is a $\gamma (D)$-set. The connection between ECD sets (whenever they exist) and the domination number of graphs was presented in several publications independently. This connection also remains in the case of digraphs as next shown.

\begin{proposition}
\label{ECD-total}If $D$ is an ECD digraph with an ECD set $S$, then $\gamma (D)=|S|$.
\end{proposition}

\begin{proof}
If $S$ is an ECD set of $D$, then $S$ is also a dominating set of $D$ and $\gamma (D)\leq |S|$
follows. On the other hand, an arbitrary vertex of $S$ has at least one neighbor in
every $\gamma (D)$-set $S'$, since $\bigcup _{v\in S'}N_{D}^+[v]=V(G)$.
Moreover, these neighbors must be different, since $\bigcup_{v\in S}N_{D}^+[v]$
form a partition of $V(D)$. Hence $\gamma(D)\geq |S|$ and the equality follows.
\end{proof}

In the definitions about domination, closed out-neighborhoods play an important role. What if we replace them by closed in-neighborhoods? Let $D$ be a digraph and $R\subseteq V(D)$. The set $R$ is called an \emph{absorbing set} of $D$ if the closed in-neighborhoods centered in vertices of $R$ cover $V(D)$, that is $V(D)=\bigcup_{v\in R}N^-_D[v]$.  Let $R$ be an absorbing set of $D$. If $N^-_D[v]\cap N^-_D[u]=\emptyset$ for every different vertices $u,v\in R$, then the set $\{N^-_D[v]: v\in R\}$ not only cover $V(D)$ but also \emph{partition} $V(D)$. In this case we say that $R$ is an \emph{efficient closed absorbing set} of $D$ (or ECA set for short). If there exists an ECA set $R$ for $D$, then $D$ is called an \emph{efficient closed absorbant digraph} (or ECA digraph for short). The minimum cardinality of an absorbing set of $D$ is called the \emph{absorbing number} of $D$ and is denoted by $\gamma_a (D)$.

It is not hard to find examples of digraphs where $\gamma(D)$ and $\gamma_a(D)$ are different. For instance, observe $K_{1,t}$, where the orientation is such that the central vertex is a source. Clearly, $\gamma (D)=1$ and $\gamma_a(D)=t$ and the difference can be arbitrary. On the other hand, there is a very strong, and also obvious, connection between dominating sets and absorbing sets as stated in the following observation. This is also the reason that with the study of ECD digraphs we also contribute to ECA digraphs (and vice versa). Therefore, we only deal in this work with ECD digraphs.

\begin{observation}
\label{ECD-ABS}A set $S$ is a dominating set of a digraph $D$ if and only if $S$ is an absorbing set of $D^-$.
\end{observation}

It is also an easy task to construct an ECD digraph $D$ if its underlying graph $G_D$ is given as well as a maximum independent set $S$ of $G_D$. For $u\in S$ and $v\in N_{G_D}(u)$ we set $(u,v)$ to be an arc whenever $v$ has no in-neighbor from $S$. Otherwise, when there already exists an in-neighbor of $v$ from $S$, we set $(v,u)$ to be an arc of $D$. In addition we orient all other edges (both end-vertices outside of $S$) arbitrary. Clearly, this procedure gives an ECD digraph $D$.

\subsection{Products of digraphs}

Let $D$ and $F$ be digraphs. Different products of digraphs $D$ and $F$ have, similar as in graphs, the set $V(D)\times V(F)$ for the set of vertices. We roughly and briefly discuss the four standard products of digraphs: the \emph{Cartesian product} $D\Box F$, the \emph{direct product} $D\times F$ , the \emph{strong product} $D\boxtimes F$ and the \emph{lexicographic product} $D\circ F$ (sometimes also denoted $D[F]$). Adjacency in different products is defined as follows.
\begin{itemize}
\item In the Cartesian product $D\Box F$ of digraphs $D$ and $F$ there exists an arc from vertex $(d,f)$ to vertex $(d',f')$ if there exists an arc from $d$ to $d'$ in $D$ and $f=f'$ or $d=d'$ and there exists an arc from $f$ to $f'$ in $F$.
\item If there is an arc from $d$ to $d'$ in $D$ and an arc from $f$ to $f'$ in $F$, then there exists an arc from $(d,f)$ to $(f',d')$ in the direct product  $D\times F$.
\item In the strong product we have $((d,f),(d',f'))\in A(D\boxtimes F)$ if ($(d,d')\in A(D)$ and $f=f'$) or ($d=d'$ and $(f,f')\in A(F)$) or ($(d,d')\in A(D)$ and $(f,f')\in A(F)$).
\item There is an arc in the lexicographic product $D\circ F$ from a vertex $(d,f)$ to a vertex $(d',f')$, whenever $(d,d')\in A(D)$ or ($d=d'$ and $(f,f')\in A(F)$).
\end{itemize}
Some examples of the above mentioned products appear in Figure \ref{Product-of-digraphs}.

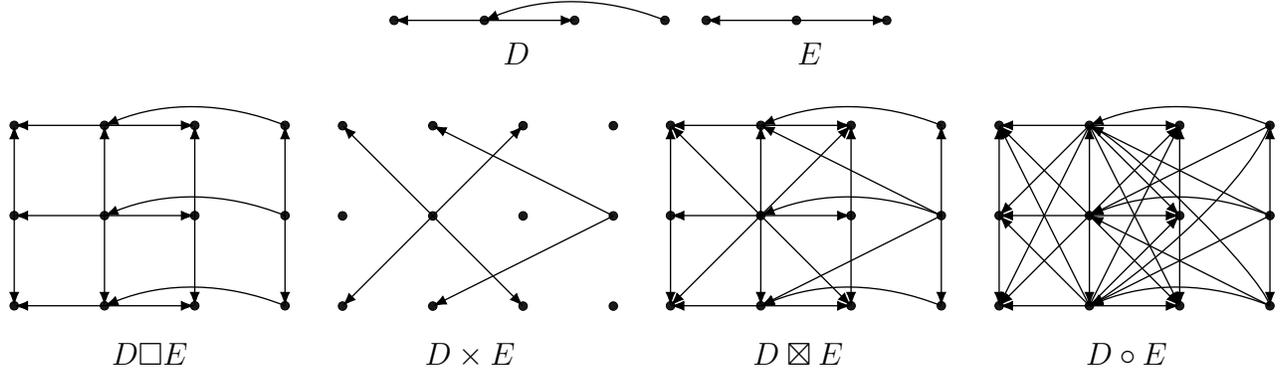
\begin{figure}[h]
\centering
\begin{tabular}{cc}
\begin{tikzpicture}[scale=.8, transform shape]
\draw[-Latex,line width=0.5pt](1.5,0)--(0,0);
\draw[-Latex,line width=0.5pt](1.5,0)--(3,0);
\draw[-Latex,line width=0.5pt] (4.5,0) .. controls (3.5,0.4) and (2.5,0.4).. (1.5,0);

\filldraw[fill opacity=0.9,fill=black]  (0,0) circle (0.07cm);
\filldraw[fill opacity=0.9,fill=black]  (1.5,0) circle (0.07cm);
\filldraw[fill opacity=0.9,fill=black]  (3,0) circle (0.07cm);
\filldraw[fill opacity=0.9,fill=black]  (4.5,0) circle (0.07cm);
\end{tikzpicture}&
\begin{tikzpicture}[scale=.8, transform shape]
\draw[-Latex,line width=0.5pt](1.5,0)--(0,0);
\draw[-Latex,line width=0.5pt](1.5,0)--(3,0);

\filldraw[fill opacity=0.9,fill=black]  (0,0) circle (0.07cm);
\filldraw[fill opacity=0.9,fill=black]  (1.5,0) circle (0.07cm);
\filldraw[fill opacity=0.9,fill=black]  (3,0) circle (0.07cm);
\end{tikzpicture}\\
$D\;\;\;$ & $\;\;\;E$
\end{tabular}

\vspace*{0.3cm}
\begin{tabular}{cccc}

\begin{tikzpicture}[scale=.8, transform shape]
\draw[-Latex,line width=0.5pt](0,1.5)--(0,3);
\draw[-Latex,line width=0.5pt](0,1.5)--(0,0);
\draw[-Latex,line width=0.5pt](1.5,1.5)--(1.5,3);
\draw[-Latex,line width=0.5pt](1.5,1.5)--(1.5,0);
\draw[-Latex,line width=0.5pt](3,1.5)--(3,3);
\draw[-Latex,line width=0.5pt](3,1.5)--(3,0);
\draw[-Latex,line width=0.5pt](4.5,1.5)--(4.5,3);
\draw[-Latex,line width=0.5pt](4.5,1.5)--(4.5,0);
\draw[-Latex,line width=0.5pt](1.5,0)--(0,0);
\draw[-Latex,line width=0.5pt](1.5,0)--(3,0);
\draw[-Latex,line width=0.5pt] (4.5,0) .. controls (3.5,0.4) and (2.5,0.4).. (1.5,0);
\draw[-Latex,line width=0.5pt](1.5,1.5)--(0,1.5);
\draw[-Latex,line width=0.5pt](1.5,1.5)--(3,1.5);
\draw[-Latex,line width=0.5pt] (4.5,1.5) .. controls (3.5,1.9) and (2.5,1.9).. (1.5,1.5);
\draw[-Latex,line width=0.5pt](1.5,3)--(0,3);
\draw[-Latex,line width=0.5pt](1.5,3)--(3,3);
\draw[-Latex,line width=0.5pt] (4.5,3) .. controls (3.5,3.4) and (2.5,3.4).. (1.5,3);

\filldraw[fill opacity=0.9,fill=black]  (0,0) circle (0.07cm);
\filldraw[fill opacity=0.9,fill=black]  (0,1.5) circle (0.07cm);
\filldraw[fill opacity=0.9,fill=black]  (0,3) circle (0.07cm);
\filldraw[fill opacity=0.9,fill=black]  (1.5,0) circle (0.07cm);
\filldraw[fill opacity=0.9,fill=black]  (1.5,1.5) circle (0.07cm);
\filldraw[fill opacity=0.9,fill=black]  (1.5,3) circle (0.07cm);
\filldraw[fill opacity=0.9,fill=black]  (3,0) circle (0.07cm);
\filldraw[fill opacity=0.9,fill=black]  (3,1.5) circle (0.07cm);
\filldraw[fill opacity=0.9,fill=black]  (3,3) circle (0.07cm);
\filldraw[fill opacity=0.9,fill=black]  (4.5,0) circle (0.07cm);
\filldraw[fill opacity=0.9,fill=black]  (4.5,1.5) circle (0.07cm);
\filldraw[fill opacity=0.9,fill=black]  (4.5,3) circle (0.07cm);

\end{tikzpicture}&
$\,$
\begin{tikzpicture}[scale=.8, transform shape]
\draw[-Latex,line width=0.5pt](1.5,1.5)--(0,0);
\draw[-Latex,line width=0.5pt](1.5,1.5)--(0,3);
\draw[-Latex,line width=0.5pt](1.5,1.5)--(3,0);
\draw[-Latex,line width=0.5pt](1.5,1.5)--(3,3);
\draw[-Latex,line width=0.5pt](4.5,1.5)--(1.5,3);
\draw[-Latex,line width=0.5pt](4.5,1.5)--(1.5,0);

\filldraw[fill opacity=0.9,fill=black]  (0,0) circle (0.07cm);
\filldraw[fill opacity=0.9,fill=black]  (0,1.5) circle (0.07cm);
\filldraw[fill opacity=0.9,fill=black]  (0,3) circle (0.07cm);
\filldraw[fill opacity=0.9,fill=black]  (1.5,0) circle (0.07cm);
\filldraw[fill opacity=0.9,fill=black]  (1.5,1.5) circle (0.07cm);
\filldraw[fill opacity=0.9,fill=black]  (1.5,3) circle (0.07cm);
\filldraw[fill opacity=0.9,fill=black]  (3,0) circle (0.07cm);
\filldraw[fill opacity=0.9,fill=black]  (3,1.5) circle (0.07cm);
\filldraw[fill opacity=0.9,fill=black]  (3,3) circle (0.07cm);
\filldraw[fill opacity=0.9,fill=black]  (4.5,0) circle (0.07cm);
\filldraw[fill opacity=0.9,fill=black]  (4.5,1.5) circle (0.07cm);
\filldraw[fill opacity=0.9,fill=black]  (4.5,3) circle (0.07cm);

\end{tikzpicture}&
$\,$
\begin{tikzpicture}[scale=.8, transform shape]
\draw[-Latex,line width=0.5pt](0,1.5)--(0,3);
\draw[-Latex,line width=0.5pt](0,1.5)--(0,0);
\draw[-Latex,line width=0.5pt](1.5,1.5)--(1.5,3);
\draw[-Latex,line width=0.5pt](1.5,1.5)--(1.5,0);
\draw[-Latex,line width=0.5pt](3,1.5)--(3,3);
\draw[-Latex,line width=0.5pt](3,1.5)--(3,0);
\draw[-Latex,line width=0.5pt](4.5,1.5)--(4.5,3);
\draw[-Latex,line width=0.5pt](4.5,1.5)--(4.5,0);
\draw[-Latex,line width=0.5pt](1.5,0)--(0,0);
\draw[-Latex,line width=0.5pt](1.5,0)--(3,0);
\draw[-Latex,line width=0.5pt] (4.5,0) .. controls (3.5,0.4) and (2.5,0.4).. (1.5,0);
\draw[-Latex,line width=0.5pt](1.5,1.5)--(0,1.5);
\draw[-Latex,line width=0.5pt](1.5,1.5)--(3,1.5);
\draw[-Latex,line width=0.5pt] (4.5,1.5) .. controls (3.5,1.9) and (2.5,1.9).. (1.5,1.5);
\draw[-Latex,line width=0.5pt](1.5,3)--(0,3);
\draw[-Latex,line width=0.5pt](1.5,3)--(3,3);
\draw[-Latex,line width=0.5pt] (4.5,3) .. controls (3.5,3.4) and (2.5,3.4).. (1.5,3);

\draw[-Latex,line width=0.5pt](1.5,1.5)--(0,0);
\draw[-Latex,line width=0.5pt](1.5,1.5)--(0,3);
\draw[-Latex,line width=0.5pt](1.5,1.5)--(3,0);
\draw[-Latex,line width=0.5pt](1.5,1.5)--(3,3);
\draw[-Latex,line width=0.5pt](4.5,1.5)--(1.5,3);
\draw[-Latex,line width=0.5pt](4.5,1.5)--(1.5,0);

\filldraw[fill opacity=0.9,fill=black]  (0,0) circle (0.07cm);
\filldraw[fill opacity=0.9,fill=black]  (0,1.5) circle (0.07cm);
\filldraw[fill opacity=0.9,fill=black]  (0,3) circle (0.07cm);
\filldraw[fill opacity=0.9,fill=black]  (1.5,0) circle (0.07cm);
\filldraw[fill opacity=0.9,fill=black]  (1.5,1.5) circle (0.07cm);
\filldraw[fill opacity=0.9,fill=black]  (1.5,3) circle (0.07cm);
\filldraw[fill opacity=0.9,fill=black]  (3,0) circle (0.07cm);
\filldraw[fill opacity=0.9,fill=black]  (3,1.5) circle (0.07cm);
\filldraw[fill opacity=0.9,fill=black]  (3,3) circle (0.07cm);
\filldraw[fill opacity=0.9,fill=black]  (4.5,0) circle (0.07cm);
\filldraw[fill opacity=0.9,fill=black]  (4.5,1.5) circle (0.07cm);
\filldraw[fill opacity=0.9,fill=black]  (4.5,3) circle (0.07cm);

\end{tikzpicture}&
$\,$
\begin{tikzpicture}[scale=.8, transform shape]

\draw[-Latex,line width=0.5pt](0,1.5)--(0,3);
\draw[-Latex,line width=0.5pt](0,1.5)--(0,0);
\draw[-Latex,line width=0.5pt](1.5,1.5)--(1.5,3);
\draw[-Latex,line width=0.5pt](1.5,1.5)--(1.5,0);
\draw[-Latex,line width=0.5pt](3,1.5)--(3,3);
\draw[-Latex,line width=0.5pt](3,1.5)--(3,0);
\draw[-Latex,line width=0.5pt](4.5,1.5)--(4.5,3);
\draw[-Latex,line width=0.5pt](4.5,1.5)--(4.5,0);
\draw[-Latex,line width=0.5pt](1.5,0)--(0,0);
\draw[-Latex,line width=0.5pt](1.5,0)--(3,0);
\draw[-Latex,line width=0.5pt] (4.5,0) .. controls (3.5,0.4) and (2.5,0.4).. (1.5,0);
\draw[-Latex,line width=0.5pt](1.5,1.5)--(0,1.5);
\draw[-Latex,line width=0.5pt](1.5,1.5)--(3,1.5);
\draw[-Latex,line width=0.5pt] (4.5,1.5) .. controls (3.5,1.9) and (2.5,1.9).. (1.5,1.5);
\draw[-Latex,line width=0.5pt](1.5,3)--(0,3);
\draw[-Latex,line width=0.5pt](1.5,3)--(3,3);
\draw[-Latex,line width=0.5pt] (4.5,3) .. controls (3.5,3.4) and (2.5,3.4).. (1.5,3);

\draw[-Latex,line width=0.5pt](1.5,1.5)--(0,0);
\draw[-Latex,line width=0.5pt](1.5,1.5)--(0,3);
\draw[-Latex,line width=0.5pt](1.5,1.5)--(3,0);
\draw[-Latex,line width=0.5pt](1.5,1.5)--(3,3);
\draw[-Latex,line width=0.5pt](1.5,0)--(0,1.5);
\draw[-Latex,line width=0.5pt](1.5,0)--(0,3);
\draw[-Latex,line width=0.5pt](1.5,0)--(3,1.5);
\draw[-Latex,line width=0.5pt](1.5,0)--(3,3);
\draw[-Latex,line width=0.5pt](1.5,3)--(0,0);
\draw[-Latex,line width=0.5pt](1.5,3)--(0,1.5);
\draw[-Latex,line width=0.5pt](1.5,3)--(3,0);
\draw[-Latex,line width=0.5pt](1.5,3)--(3,1.5);

\draw[-Latex,line width=0.5pt](4.5,1.5)--(1.5,3);
\draw[-Latex,line width=0.5pt](4.5,1.5)--(1.5,0);
\draw[-Latex,line width=0.5pt](4.5,3)--(1.5,1.5);
\draw[-Latex,line width=0.5pt](4.5,3).. controls (4,1.8) and (2,0.3)..(1.5,0);
\draw[-Latex,line width=0.5pt](4.5,0)--(1.5,1.5);
\draw[-Latex,line width=0.5pt](4.5,0).. controls (4.3,0.6) and (2.2,2.6)..(1.5,3);

\filldraw[fill opacity=0.9,fill=black]  (0,0) circle (0.07cm);
\filldraw[fill opacity=0.9,fill=black]  (0,1.5) circle (0.07cm);
\filldraw[fill opacity=0.9,fill=black]  (0,3) circle (0.07cm);
\filldraw[fill opacity=0.9,fill=black]  (1.5,0) circle (0.07cm);
\filldraw[fill opacity=0.9,fill=black]  (1.5,1.5) circle (0.07cm);
\filldraw[fill opacity=0.9,fill=black]  (1.5,3) circle (0.07cm);
\filldraw[fill opacity=0.9,fill=black]  (3,0) circle (0.07cm);
\filldraw[fill opacity=0.9,fill=black]  (3,1.5) circle (0.07cm);
\filldraw[fill opacity=0.9,fill=black]  (3,3) circle (0.07cm);
\filldraw[fill opacity=0.9,fill=black]  (4.5,0) circle (0.07cm);
\filldraw[fill opacity=0.9,fill=black]  (4.5,1.5) circle (0.07cm);
\filldraw[fill opacity=0.9,fill=black]  (4.5,3) circle (0.07cm);

\end{tikzpicture}\\[0.2cm]

$D\Box E$&$D\times E$&$D\boxtimes E$&$D\circ E$
\end{tabular}
\caption{The digraphs $D$ and $E$, and their Cartesian, direct, strong and lexicographic products.}
\label{Product-of-digraphs}
\end{figure}

First three mentioned products are clearly commutative, while lexicographic product is not. On the other hand, all four of them are associative (see Chapter 32 from \cite{ImKl} - also for additional information about products of digraphs).

We next use the symbol $*$ for any of the four standard products $\{\Box, \times, \boxtimes, \circ \}$. For a fixed
$f\in V(F)$ we call $D^{f}=\{(d,f)\in V(D*F):d\in V(D)\}$ a $D$-\emph{layer through} $f$
in $D*F$. Symmetrically, an $F$-layer $^{d}\!F$ through $d$ is defined for a fixed $d\in V(D)$. Notice that the subdigraph of $D*F$ induced by a $D$-layer or by an $F$-layer is isomorphic to $D$ or to $F$, respectively, for the Cartesian product, for the strong product and for the lexicographic product. In the case of direct products loops play an important role. The subdigraph induced by the layer $D^{f}$, or by the layer $^{d}\!F$, is an empty digraph on $|V(D)|$ vertices or $|V(F)|$ vertices, respectively, if there is no loop in $f$ and in $d$, respectively. If we have $(d,d)\in A(D)$ and $(f,f)\in A(F)$, then $F[^d\!F]$ and $D[D^f]$, respectively, are isomorphic to $F$ and $D$, respectively.

The map $p_{D}:V(D*F)\rightarrow V(D)$ defined by $p_{D}((d,f))=d$ is called a \emph{projection map onto} $D$. Similarly, we define
$p_{F}$ as the \emph{projection map onto} $F$. Projections are defined as
maps between vertices, but frequently it is more convenient to see them like maps between digraphs. In this case we observe the subdigraphs induced by
$B\subseteq V(D\circ F)$ and $p_{X}(B)$ for $X\in \{D,F\}$. Notice that in Cartesian and strong product the arcs project either to arcs (with the same orientation) or to a vertex. In the case of direct product arcs always project to arcs (with the same orientation). The projection $p_D$ maps arcs into arcs (with the same orientation) or vertices in lexicographic product $D\circ F$, while in the same product arcs are mapped with $p_F$ into vertices, into arcs with the same orientation, into arcs with different orientation or into two vertices which do not form an arc in $F$.

It is easy to see that closed out-neighborhoods in strong product of digraphs satisfy that

\begin{equation}
 N_{D\boxtimes F}^+[(d,f)]=N_{D}^+[d]\times N_{F}^+[f] \label{strongneig}
\end{equation}
and for lexicographic product of digraphs we have
\begin{equation}
 N_{D\circ F}^+[(d,f)]=N_{D}^+(d)\times V(F) \cup \{d\}\times N_{F}^+[f]. \label{lexneig}
\end{equation}
This two facts are the reason that we can present a complete characterization of ECD digraphs among strong and lexicographic products with the properties of factors in the last section.

For the direct product of digraphs we will use the following property on out- and in-degrees
\begin{equation}
 \delta^+_{D\times F}((d,f))=\delta_D^+(d)\delta_{F}^+(f)\;\; {\rm and }\;\; \delta^-_{D\times F}((d,f))=\delta_D^-(d)\delta_{F}^-(f). \label{dirdeg}
\end{equation}

With all the terminologies, notations and definitions described till this point, we are then able to present our results concerning ECD digraphs among Cartesian, direct, strong and lexicographic product of digraphs, which we do in next sections.



\section{Cartesian product}

In this section we completely describe all ECD digraphs among Cartesian products when one factor is either a cycle without sinks or a star with such an orientation that the central vertex is the source. This approach is partially inspired by \cite{KPRT} where similar approach was taken for efficient open domination Cartesian products graphs. With one fix factor it is possible to concentrate on the properties of the other factor to completely describe when the product is an ECD digraph. In order to cover all possible configurations of ECD digraphs among Cartesian product of digraphs whether one factor is a cycle, we need to describe three families of digraphs as follows.

Let $D'$ be an arbitrary digraph and let $\Pi_1=\{W_1,\ldots,W_p\}$ and $\Pi_2=\{Z_1,\ldots,Z_r\}$ be any two partitions of $V'=V(D')$. By using these partitions we define a new digraph $D_1$ as follows. Its vertex set is $V(D_1)=V'\cup W\cup Z$, where $W=\{w_1,\ldots,w_p\}$ and $Z=\{z_1,\ldots,z_r\}$ are two additional disjoint set of vertices. The arc set is
$$A(D_1)=A(D')\cup \{(w_i,v):v\in W_i,i\in\{1,\ldots,p\}\}\cup \{(z_i,v):v\in Z_i,i\in\{1,\ldots,r\}\}\cup B,$$
where $B$ is an arbitrary set of arcs from vertices in $V'$ to vertices in $W\cup Z$. We denote the family of digraphs that contains the trivial digraph on one vertex without a loop, and all possible digraphs $D_1$ obtained from an arbitrary starting digraph $D'$ by $\mathcal{D}_1$. 

We remark that any digraph $D_1\in \mathcal{D}_1$ can be alternatively described as a digraph where $W,Z$ and $V'$ partition $V(D_1)$ such that $W$ is an ECD set in the digraph induced by $W\cup V'$; $Z$ is an ECD set in the digraph induced by $Z\cup V'$; and there are no arcs in any direction between $W$ and $Z$. This means, among other facts, that $W\rightarrow V'$, $Z\rightarrow V'$ and $W\nleftrightarrow Z$.

Suppose now that we can partition a digraph $D$ into three sets $U_1,U_2$ and $U_3$ such that $U_1$ is an ECD set in $D[U_1\cup U_2]$, $U_2$ is an ECD set in $D[U_2\cup U_3]$, and $U_3$ is an ECD set in $D[U_1\cup U_3]$, and no other arc exist in $D$. Clearly, each set $U_1$, $U_2$, and $U_3$ induces a digraph without arcs as they are ECD sets for some digraph. Also, every vertex $v$ from $D$ has $\delta_D^-(v)=1$, which means that no isolated vertices exists in $D$. By $\mathcal{D}_2$, we denote the family of all possible digraphs with such described structure.

Finally, a digraph $D$ is in $\mathcal{D}_3$ if its vertex set can be partitioned into two subsets, such that the first one induces a digraph $D_1$ from $\mathcal{D}_1$, and the second induces a digraph $D_2$ from $\mathcal{D}_2$. In addition, there exists some arcs from $V'\subset V(D_1)$ to $V(D_2)$.

We are now ready to completely describe when the Cartesian product of any digraph $D$ with a directed cycle with no sources is an ECD digraph. While it is clear that $D\Box C_1^0$ is an ECD digraph if and only if $D$ is an ECD graph, we may restrict to the case when $k\geq 2$.

\begin{theorem}\label{ecdcartcycle}
Let $D$ be a digraph and let $k\geq 2$ be an integer. The Cartesian product $D\Box C_k^0$ is an ECD digraph if and only if for every component $C$ of $D$ we have
\begin{itemize}
\item[{\rm (i)}] $C\in \mathcal{D}_1$ and $k$ is even, or

\item[{\rm (ii)}] $C\in \mathcal{D}_2$ and $k$ is a multiple of 3, or

\item[{\rm (iii)}] $C\in \mathcal{D}_3$ and $k$ is a multiple of 6.
\end{itemize}
\end{theorem}

\begin{proof} Suppose that $V(C_k^0)=\{v_1,\ldots,v_k\}$ where $A(C_k^0)=\{(v_k,v_1)\}\cup \{(v_i,v_{i+1}):i\in\{1,\ldots,k-1\}\}$. Let $C$ be any component of $D$. We first assume that $k=2\ell$ for some positive integer $\ell$  and let $C\in \mathcal{D}_1$. Hence, $C=D_1$ for some digraph obtained from a digraph $D'$ together with the sets of vertices $W=\{w_1,\ldots,w_p\}$ and $Z=\{z_1,\ldots,z_r\}$. We set $S^i_1=W\times \{v_{2i-1}\}$ and $S^i_2=Z\times \{v_{2i}\}$ for every $i\in\{1,\ldots,\ell \}$. We will show that $S_1=\bigcup_{i=1}^{\ell}(S_1^i\cup S_2^i)$ is an ECD set of $C\Box C_k^0$. Every vertex from $V'\times \{v_{2i-1}\}$ is adjacent from at least one vertex from $S_1^{i}$ because $W$ is an ECD set in the digraph induced by $W\cup V'$. Similarly, every vertex from $V'\times \{v_{2i}\}$ is adjacent from at least one vertex from $S_2^i$ because $Z$ is an ECD set in the digraph induced by $Z\cup V'$. By the properties of Cartesian product, every vertex from $W\times \{v_{2i}\}$ is adjacent from exactly one vertex from $S_1^{i}$, and every vertex from $Z\times \{v_{2i+1}\}$ is adjacent from exactly one vertex from $S_2^{i}$ (indices are understood modulo $k$). Therefore, every vertex is either in $S_1$ or is adjacent from a vertex from $S_1$ and so, closed out-neighborhoods centered in $S_1$ form a cover of $V(C\Box C_k^0)$.

We still need to show that any two different closed out-neighborhoods centered in vertices from $S_1$ have an empty intersection. If this is not the case, then two vertices from $S_1$ are either both in $S_1^i$, or both in $S_2^i$, or one in $S_1^i$ and other in $S_2^i$, or one in $S_2^i$ and another in $S_1^{i+1}$ for every $i\in\{1,\ldots,\ell \}$. In the first two cases we obtain a contradiction with the properties of Cartesian product or with the fact that $p_C(S_1^i)$ and  $p_C(S_2^i)$ are ECD sets for $C[W\cup V']$ and $C[Z\cup V']$, respectively. In the last two cases we obtain a contradiction with the fact that there are no arcs between $p_C(S_1^i)$ and $p_C(S_2^i)$ in any direction. Hence, $S_1$ is and ECD set, and therefore, $C\Box C_k^0$ is an ECD digraph.

Let know $k=3\ell$ and $C\in \mathcal{D}_2$, and in this sense, let be $U_1,U_2$ and $U_3$ as previously described. We will show that the set $S_2=S^1\cup S^2\cup S^3$ is an ECD set of $C\Box C_k^0$, where $S^1=U_1\times \{v_{3i+1}:i\in \{0,\ldots,\ell-1\}\}$, $S^2=U_2\times \{v_{3i+2}:i\in \{0,\ldots,\ell-1\}\}$ and $S^3=U_3\times \{v_{3i}:i\in \{1,\ldots,\ell\}\}$.
We first show that the closed neighborhoods centered in vertices from $S_2$ cover $V(C\Box C_k^0)$. By the definition of the Cartesian product, we have that the vertices from $U_2\times \{v_{3i+1}:i\in \{0,\ldots,\ell-1\}\}$ and from $U_1\times \{v_{3i+2}:i\in \{0,\ldots,\ell-1\}\}$ have an in-neighbor in $S^1$. Similarly, those vertices from $U_3\times \{v_{3i+2}:i\in \{0,\ldots,\ell-1\}\}$ and from $U_2\times \{v_{3i}:i\in \{1,\ldots,\ell\}\}$ have an in-neighbor in $S^2$. Finally, the vertices from $U_1\times \{v_{3i}:i\in \{1,\ldots,\ell\}\}$ and from $U_3\times \{v_{3i+1}:i\in \{0,\ldots,\ell-1\}\}$ have an in-neighbor in $S^3$, and every vertex from $C\Box C_k^0$ is contained in at least one closed neighborhood centered in a vertex from $S_2$.

Next we show that all these different closed neighborhoods have pairwise empty intersection. Clearly, by the orientation of $C_k^0$ converse may occur only if closed neighborhoods have centers in the same layer of $C^{v_i}$, or in-neighboring layers $C^{v_i}$ and $C^{v_{i+1}}$ for some $i\in \{1,\ldots,k\}$ (if $i=k$, then $i+1$ must be considered modulo $k$). The first case is not possible because $U_j$, $j\in\{1,2,3\}$, is an ECD set in $D[U_j\cup U_{j+1}]$, and there is no arc from $U_j$ to $U_{j+2}$ (subscripts must be considered modulo 3 here). The second case is not possible by the orientation of the arc $(v_i,v_{i+1})$, and again, since there is no arc from $U_j$ to $U_{j+2}$ in $C^{v_{i+1}}$. Therefore, $S_2$ is an ECD set and $C\Box C_k^0$ is an ECD graph.

To end this implication let $k=6\ell$ and $C\in \mathcal{D}_3$. Let $C_1$ and $C_2$ be subdigraphs of $C$ such that $C_1\in \mathcal{D}_1$ and that $C_2\in \mathcal{D}_2$, respectively. As $k$ is even and also a multiple of three, $C_1\Box C_k^0$ and $C_2\Box C_k^0$ are ECD digraphs with ECD sets $S_1$ and $S_2$, respectively, as described. The set $S=S_1\cup S_2$ is and ECD set of $C\Box C_k^0$ because every arc between $C_1$ and $C_2$ starts in $V'$ and end in $V(C_2)$ and have therefore no influence on closed out-neighborhoods centered in $S$ as $S\cap (V'\times V(C_k^0))=\emptyset$. Thus, $C\Box C_k^0$ is an ECD set.

Since the component $C$ was arbitrarily chosen, $C\Box C_k^0$ is an ECD digraph for every component $C$ from $D$ and therefore, also $D\Box C_k^0$ is an ECD digraph.

Let now $D\Box C_k^0$ be an ECD digraph. Clearly, for every component $C$ of $D$, $C\Box C_k^0$ is also an ECD digraph. Let $S$ be an ECD set of $C\Box C_k^0$ and let $S_i=S\cap C^{v_i}$. Suppose that there exists $i\in\{1,\ldots,k\}$ (by symmetry we may assume that $i=1$) such that $S_1=C^{v_1}$. Clearly, in such a case $C$ is a one vertex digraph and $C\in \mathcal{D}_1$ (also note that every $S_i$ is formed by only one vertex). By the properties of Cartesian product, the set $S_1$ efficiently closed dominates $C^{v_2}$. Therefore, $S_3=C^{v_3}$ and $S_3$ efficiently closed dominates $C^{v_4}$. Continuing in this way, we see that $k$ must be an even number which finishes this case.

We may assume now that $C$ is not an isolated vertex and that $S_i\neq C^{v_i}$ for every $i\in\{1,\ldots,k\}$. Let $W=p_C(S_1)$, let $Z$ be the set of all vertices from $V(C)$ such that $W\nrightarrow Z$ and let $V'=V(C)-(W\cup Z)$. By the definition of sets $W,Z$ and $V'$, we have $W\rightarrow V'$. If $Z=\emptyset$, then $S_k=\emptyset$, and therefore $S_{k-1}=C^{v_{k-1}}$, which is a contradiction. Thus $Z\neq \emptyset$. The vertices from $Z\times \{v_1\}$ can be covered only by vertices from $Z\times \{v_k\}$ and arcs between these two sets form a matching. Hence, $Z\times \{v_k\}\subseteq S_k$. If $S_k\neq Z\times \{v_k\}$, then we have an arc between two vertices of $S$, or a vertex from $V'\times \{v_1\}$ is dominated twice from vertices in $S$, and both choices are not possible. Therefore, $S_k=Z\times \{v_k\}$. Now we move to $C^{v_k}$ and $S_k$. For this we separate three cases and use the notation $Z_k$ for the set $N^+_{C\Box C_k^0}(S_k)\cap C^{v_k}$, which are all out-neighbors of $S_k$ in $C^{v_k}$.\medskip

\noindent \textbf{Case 1: }$Z_k\subseteq V'\times \{v_k\}$

\noindent Suppose, with a purpose of contradiction, that there exists a vertex $(x,v_k)\in (V'\times \{v_k\})-Z_k$. Hence, the set $(W\times \{v_k\})\cup \{(x,v_k)\}$ is not dominated by $S_k$ and by the properties of Cartesian product $\left((W\times \{v_{k-1}\})\cup \{(x,v_{k-1})\}\right)\subseteq S_{k-1}$. This is not possible because $(x,v_{k-1})$ has an in-neighbor in $W\times  \{v_{k-1}\}$ as $(x,v_1)$ has an in-neighbor in $S_1=W\times \{v_1\}$. Thus, $V'\times \{v_k\}=Z_k$ and $W\times \{v_{k-1}\}=S_{k-1}$ by the properties of Cartesian product. If $k=2$, then we are done as $S_{k-1}=S_1$ and $C=D_1\in \mathcal{D}_1$ by the chosen notation. If $k>2$, then we notice that $C^{v_{k-1}}$ has the same structure as $C^{v_1}$. By the same arguments we get that $C^{v_{k-2}}$ has the same structure as $C^{v_k}$ and $C^{v_{k-3}}$ has the same structure as $C^{v_1}$. Repeating that step, we obtain that $C^{v_{2i+1}}$ has the same structure as $C^{v_1}$ and that $C^{v_{2i}}$ has the same structure as $C^{v_k}$ for every $i\in\{1,\ldots,\ell-1\}$. Therefore, $k=2\ell$ and, again by the chosen notation, we have that $C=D_1\in \mathcal{D}_1$. \medskip

\noindent \textbf{Case 2: }$Z_k\subseteq W\times \{v_k\}$

\noindent Again, in order to get a contradiction, suppose that there exists a vertex $(x,v_k)\in (W\times \{v_k\})-Z_k$. In addition we change notation to $U_1=W$, $U_2=V'$ and $U_3=Z$. If $\delta^+_C(x)>0$, then there exists an arc $(x,y)\in A(C)$ for some $y\in p_C(U_2)$. Now, we observe that $\{(x,v_{k-1})\}\cup (U_2\times \{v_{k-1})\}\subseteq S_{k-1}$, by properties of Cartesian product and also, since $(x,v_k)$ must be dominated by $S$. This yields a contradiction because there exists an arc $((x,v_{k-1}),(y,v_{k-1}))$ between two vertices of $S_{k-1}$ which is not possible. Otherwise, if $\delta^+_C(x)=0$, then $\delta^-_C(x)>0$, as $C$ is not an isolated vertex. Clearly, the in-neighbor $y$ of $x$ in $C$ must be in $U_2$, because $x$ is in $U_1$, which is an independent set of vertices and there exists no in-neighbor of $x$ in $U_3$, since $(x,v_k)\in (U_1\times \{v_k\})-Z_k$. Again, $\{(x,v_{k-1})\}\cup (U_2\times \{v_{k-1})\}\subseteq S_{k-1}$ by properties of Cartesian product and also, because $(x,v_k)$ must be dominated by $S$. Consequently, we have a similar contradiction with the existence of an arc $((y,v_{k-1}),(x,v_{k-1}))$ between two vertices of $S_{k-1}$.

Therefore, $U_1\times \{v_k\}=Z_k$ and $U_2\times \{v_{k-1}\}\subseteq S_{k-1}$ in order to cover all vertices from $U_2\times \{v_k\}$. Even more, $U_2\times \{v_{k-1}\}=S_{k-1}$, since otherwise a vertex from $(U_1\cup U_3)\times \{v_{k}\}$ would be in two different out-neighborhoods centered in $S$, which is not possible. Now, suppose that there exists a vertex $x\in U_3$ such that $(x,v_{k-1})$ is not covered by $S_{k-1}$. This leads to $\delta^-_C(x)=0$ due to the definition of $U_3$, which implies that $\delta^+_C(x)>0$. Let $y\in U_1$ be an out-neighbor of $x$ in $C$. To cover the remaining parts of $C^{v_{k-1}}$, we have $\{(x,v_{k-2})\}\cup(U_1\times \{v_{k-2}\}\subseteq S_{k-2}$. Hence, the arc $((x,v_{k-2}),(y,v_{k-2}))$ produces another contradiction as both end-vertices $(y,v_{k-2})$ and $(x,v_{k-2})$ belong to $S_{k-2}$. Therefore, every vertex from $U_3\times \{v_{k-1}\}$ is dominated by some vertex from $U_2\times \{v_{k-1}\}$. Finally, $S_{k-2}=U_1\times\{v_{k-2}\}$ since all vertices from $U_1\times\{v_{k-1}\}$ are dominated by $S$ as well.

Notice that the structure of the layer $C^{v_1}$ implies that $U_1$ is an ECD set in the digraph $C[U_1\cup U_2]$. Similarly, from the layer $C^{v_{k}}$ we deduce that $U_3$ is an ECD set in the digraph $C[U_1\cup U_3]$, and finally, that $U_2$ is an ECD set in the digraph $C[U_2\cup U_3]$ can be seen from the layer $C^{v_{k-1}}$. Therefore, $C\in \mathcal{D}_2$ by the used notation. If $k=3$, then $C^{v_{k-2}}=C^{v_1}$ and we are done. If $k>3$, then we continue with the same pattern and we can see that $S_i=U_1\times \{v_i\}$ whenever $i=3j-2$, that $S_i=U_2\times \{v_i\}$ whenever $i=3j-1$ and that $S_i=U_2\times \{v_i\}$ whenever $i=3j$ for some $j\in\{1,\ldots,\ell\}$. Then, clearly, $k=3\ell$ which completes this case.\medskip

\noindent \textbf{Case 3: }$Z_k\cap (W\times \{v_k\})\neq \emptyset$ and $Z_k\cap (V'\times \{v_k\})\neq \emptyset$

\noindent	Let us first introduce some notation. Let $A_2=p_C(Z_k\cap (W\times \{v_k\}))$	and $A_3=p_C(Z_k\cap (V'\times \{v_k\}))$. In $A_6$ are all sinks from $Z$ and all vertices from $Z$  that are adjacent to vertices in $A_2$ and $A_5=Z-A_6$. In addition let $A_1=W-A_2$ and $A_4=V'-A_3$. From these definitions we have the following:
\begin{itemize}
  \item $A_1\nleftrightarrow A_2$ because $(A_1\cup A_2)\times\{v_1\}=S_1$,
  \item $A_5\nleftrightarrow A_6$ because $(A_5\cup A_6)\times\{v_1\}=S_k$,
  \item $A_1,A_2\nrightarrow A_5,A_6$ because $A_1\cup A_2=W$ and $A_5\cup A_6=Z$,
  \item $A_5\rightarrow A_3$ and $A_6\rightarrow A_2$ by the definitions of $A_5$ and $A_6$,
  \item $A_5\nrightarrow A_2$ and $A_6\nrightarrow A_3$ by the previous item,
  \item $A_5\nrightarrow A_4$ by the definition of $A_4$,
  \item $A_6\nrightarrow A_1$ by the definition of $A_2$ in $S_k$,
  \item $A_6\nrightarrow A_4$ and $A_5\nrightarrow A_1$, both by the definition of $A_2$ and $A_3$ in $S_k$.
\end{itemize}
Now, in order to dominate $A_1\times \{v_k\}$ (and respectively $A_4\times \{v_k\}$) by $S$, it must happen $A_1\times \{v_{k-1}\}\subset S_{k-1}$ (and respectively $A_4\times \{v_{k-1}\}\subset S_{k-1}$). This brings immediately $A_1\nleftrightarrow A_4$, and therefore, $A_2\rightarrow A_4$ because of $S_1$. This, in addition, brings that out-neighbors from $A_1$ must be in $A_3$. If $A_1$ does not dominate every vertex from $A_3$, then we obtain the same contradiction as in Case 1. Hence we have that $A_1\rightarrow A_3$, which leads to that $A_4\nrightarrow A_3$, otherwise $S_{k-1}$ is not a subset of the ECD set $S$. Another consequence of $A_1\rightarrow A_3$ is that $A_2\nrightarrow A_3$, for otherwise, $S_1$ would dominate some vertices of $A_3$ more than once.

If $A_4$ does not cover all vertices from $A_6$, then some vertices of $A_6\times \{v_{k-2}\}$ must be in $S_{k-2}$. This is a contradiction because $(A_2\times \{v_{k-2}\})\subset S_{k-2}$ and $A_6\rightarrow A_2$. Therefore, we have $A_4\rightarrow A_6$. Finally, we show that $A_4\nrightarrow A_5$. Suppose on the contrary that a vertex $(x,v_{k-1})$ from $A_5\times \{v_{k-1}\}$ is dominated from $A_4\times \{v_{k-1}\}$. Note that the vertex $(x,v_{k-2})$ is therefore not in $S_{k-2}$ and $(x,v_{k-3})$ must be in $S_{k-3}$. Recall that every vertex from $A_5$ has an out-neighbor, say $y$. All out-neighbors from $A_5$ are in $A_3$ and $(y,v_{k-3})$ is covered from $(x,v_{k-3})\in S_{k-3}$. On the other hand, $A_1\times \{v_{k-3}\}\subset S_{k-3}$ in order to cover vertices from $A_1\times \{v_{k-2}\}$. Because $A_1\rightarrow A_3$, the vertex $(y,v_{k-3})$ has at least two in-neighbors in $S_{k-3}$, which is a contradiction. Thus, $A_4\nrightarrow A_5$ holds.

As a consequence, by this notation, we can see that $A_1,A_3$ and $A_5$ have the same role as $W,V'$ and $Z$, respectively, in Case 1. By the same reason, as therein, it happens $C[A_1\cup A_3\cup A_5]$ induces a digraph from $\mathcal{D}_1$. Similarly, $A_2,A_4$ and $A_6$ play the role of $U_1,U_2$ and $U_3$, respectively, in Case 2. By following the same lines as in Case 2, we note that $C[A_2\cup A_4\cup A_6]$ induces a digraph from $\mathcal{D}_2$. Since $C$ is a connected component, this yields an existence of some arcs between $C[A_1\cup A_3\cup A_5]$ and $C[A_2\cup A_4\cup A_6]$. The only possibility are arcs from some vertices of $A_3=V'$ to some vertices of $A_2\cup A_4\cup A_6$. But in this case $C\in \mathcal{D}_3$. The proof is finished, since $k$ must be both a multiple of 2 as $C[A_1\cup A_3\cup A_5]\in \mathcal{D}_1$ and a multiple of 3 as $C[A_1\cup A_3\cup A_5]\in \mathcal{D}_2$. 
\end{proof}

While it could probably be algorithmically difficult to recognize those digraphs that belong either to $\mathcal{D}_1$ or to $\mathcal{D}_2$, it is very easy to construct such digraphs as explained in their definition. Also, the variety of digraphs in $\mathcal{D}_1$ is wider than that in $\mathcal{D}_2$, due to the in-degree condition of every vertex of a digraph from $\mathcal{D}_2$.

We next continue by describing a family of digraphs which is crucial for a digraph $F\Box K_{1,t}$ to be an ECD digraph. Let $D$ be an arbitrary digraph. We first partition $V(D)$ into nonempty sets $W_1,\ldots,W_p$. We then define a digraph $D_p$ from $D$ as follows. If $A(D)=\emptyset$ ($D$ has no arcs), then $D_p=D$. Otherwise, $V(D_p)=V(D)\cup \{w_1,\ldots,w_p\}$, where $\{w_1,\ldots,w_p\}$ is a set of external vertices, and
$$A(D_p)=A(D)\cup \{(w_i,v):v\in W_i,i\in\{1,\ldots,p\}\}\cup B,$$
where $B$ is any subset of the set $\{(v,w_i):v\in V(D),i\in\{1,\ldots,r\}\}$. Notice that $D_p$ is an ECD digraph with an ECD set $W=\{w_1,\ldots,w_p\}$.

Further, let $Z_1,\ldots,Z_r$ be a partition of $V(D_p)$ into nonempty sets. We define a new digraph $D_{p,r}$ as follows. If $A(D_p)=\emptyset $, then $D_{p,r}=D_p$. Otherwise, $V(D_{p,r})=V(D_p)\cup \{z_1,\ldots,z_r\}$, where $\{z_1,\ldots,z_r\}$ is other set of external vertices, and
$$A(D_{p,r})=A(D_p)\cup \{(z_i,u):u\in Z_i,i\in\{1,\ldots,r\} \}\cup B'$$
where $B'$ is any subset of the set $\{(v,z_i):v\in V(D),i\in\{1,\ldots,r\}\}$. By the same reason as above, the set $Z=\{z_1,\ldots,z_r\}$ is an ECD set of $D_{p,r}$, which is therefore an ECD digraph as well. Notice that $\delta^-_{D_{p,r}}(w_j)=1$ for every $j\in\{1,\ldots,p\}$.

Given a digraph $D$, by $\mathcal{D}_D$ we denote the family of all possible digraphs $D_{p,r}$ constructed from $D$, and by $\mathcal{D}_0$ we represent the union of all families $\mathcal{D}_D$ over any arbitrary digraph $D$. We can summarize the elements of $\mathcal{D}_D$ as digraphs that have an ECD set $S$ and the digraph induced by $V(D)-S$ is an ECD digraph again with an ECD set $S'$ such that there exists no arc that starts in $S'$ and ends in $S$.

\begin{theorem}\label{ecdcartsun}
Let $F$ be a digraph and $K_{1,t}$ is oriented in such a way that its central vertex is a source. The Cartesian product $F\Box K_{1,t}$ is an ECD digraph if and only if $F\in \mathcal{D}_0$.
\end{theorem}

\begin{proof} 
Assume first $F\in \mathcal{D}_0$ and let $F=D_{p,r}$. Also, let $D_p$, $W$ and $Z$ be as described before. Moreover, let $v\in V(K_{1,t})$ be the central vertex of $K_{1,t}$, which is a source. This means that all leaves $\{u_1,\ldots,u_t\}$ of $K_{1,t}$ are sinks. We claim that the set
$$S=\left(Z\times \{v\}\right) \cup \bigcup_{i=1}^t\left(W\times \{u_i\}\right)$$
is an ECD set of $F\Box K_{1,t}$.

Every vertex from $(V(F)-Z)\times \{v\}$ is dominated by exactly one vertex of $Z\times \{v\}$ because $Z$ is an ECD set for $D_{p,r}=F$. In addition, the set $Z\times \{v\}$ also dominates every vertex from $Z\times \{u_i\}$, for every $i\in\{1,\ldots,t\}$, exactly once, with the arc $((z_j,w),(z_j,u_i))$ for every $j\in\{1,\ldots,r\}$. Any vertex $(x,u_i)$ for $x\in V(F)-(W\cup Z)$ and $i\in\{1,\ldots,t\}$ has an in-neighbor in $W\times \{u_i\}$ because $W$ is an ECD set of $D_p$. Therefore, each vertex outside of $S$ has at least one in-neighbor in $S$ and $V(F\Box K_{1,t})\subseteq \bigcup_{y\in S}N_{F\Box K_{1,t}}[y]$.

We need to show that different closed neighborhoods centered in vertices from $S$ are pairwise disjoint. Suppose on the contrary, that $N^+_{F\Box K_{1,t}}[y]\cap N^+_{F\Box K_{1,t}}[y']$ is nonempty for some $y,y'\in S$, where $y=(y_1,y_2)\neq (y'_1,y'_2)=y'$. If $y,y'\in Z\times \{v\}$, then $y_1\neq y'_1$, and we have a contradiction with $Z$ being an ECD set of $F=D_{p,r}$. If $y\in Z\times \{v\}$ and $y'\in W\times \{u_i\}$ for some $i\in\{1,\ldots,t\}$, then $N^+_{F\Box K_{1,t}}[y]\cap Z\times \{u_i\}=\{(y_1,u_i)\}$ and $N^+_{F\Box K_{1,t}}[y']\subseteq F^{u_i}-\left(Z\times \{u_i\}\right)$, which is a contradiction, because $(y_1,u_i)\notin N^+_{F\Box K_{1,t}}[y']$. Let now $y,y'\in W\times \{u_i\}$ for some $i\in\{1,\ldots,t\}$. In this case $y_1\neq y'_1$ and we have a contradiction with $W$ being an ECD set of $D_p$, since
$$N^+_{F\Box K_{1,t}}(y),N^+_{F\Box K_{1,t}}(y')\subseteq F^{u_i}-\left(\left(Z\cup W\right)\times \{u_i\}\right).$$
Finally, if $y\in W\times \{u_i\}$ and $y'\in W\times \{u_j\}$ for some different $i,j\in\{1,\ldots,t\}$, then we have an immediate contradiction as $u_i$ and $u_j$ are both sinks in $K_{1,t}$. Therefore, every two distinct closed neighborhoods centered in $S$ have empty intersection, which means $S$ is an ECD set and so, $F\Box K_{1,t}$ is an ECD digraph.

Conversely, we assume that the Cartesian product $F\Box K_{1,t}$ is an ECD digraph with an ECD set $S$. We use the same notation for $K_{1,t}$, where the central vertex $v$ is the source and all leaves $u_1,\ldots,u_t$ are sinks. Let also $Z^v=S\cap F^v$ and $Z=p_F(Z^v)$. If $Z^v=F^v$, then $F$ must be an empty digraph (recall that empty digraphs belong to $\mathcal{D}_D$), and we can understand $F$ as a digraph $D_{p,r}=D_p=D$. Thus, clearly $F\in \mathcal{D}_D$ and therefore, also $F\in \mathcal{D}_0$. So, from now on, we may assume that $F$ contains some arcs. Notice that $Z^v$ is nonempty because $v$ is the source of $K_{1,t}$ and all in-neighbors from vertices of $F^v$ are also of $F^v$. On the other hand, every vertex of $F^v-Z^v$ has exactly one in-neighbor in $Z^v$, since $S$ is an ECD set. Thus, as $F$ is isomorphic to the subdigraph induced by $F^v$, we have that $Z$ is an ECD set of $F$, or equivalently, $F$ is an ECD digraph. It hence remains to prove that also $F^v-Z$ induces an ECD digraph.

To this end, we next consider the subgraph of $F\Box K_{1,t}$ induced by $(V(F)-Z)\times \{u_1\}=F_p^{u_1}$. The in-neighbors from vertices of $F_p^{u_1}$ are either in $F_p^{u_1}$ or in $Z\times \{u_1\}$. Based on the fact that every vertex of $Z\times \{u_1\}$ is dominated by a vertex from $Z\times \{v\}$ which is a subset of $S$, it must happen that the subdigraph $F'=F\Box K_{1,t}[F_p^{u_1}]$ will be an ECD digraph itself. If $F'$ is without arcs, then we can understand $F$ as a digraph $D_{p,r}$ obtained from $D_p=F-Z=D$ and so, $F\in \mathcal{D}_D$, which therefore means $F\in\mathcal{D}_0$. In concordance, we may assume that $F'$ contains some arcs. Since $S$ is an ECD set of $F\Box K_{1,t}$, the intersection $W^{u_1}=S\cap F^{u_1}$ is nonempty and, moreover, $W^{u_1}$ is an ECD set of $F'$. Let $W=p_F(W^{u_1})$, let $B'$ be the set of arcs of $F$ that start in a vertex in $V(F)-Z$ and end at some vertex in $Z$, and let $B$ be the sets of all arcs that start in $V(F)-(Z\cup W)$ and end in $W$. In this sense, we take the digraph $D$ induced by the set of vertices $V(F)-(Z\cup W)$. Hence, notice that a digraph $D_p$ can be obtained from $D$ by adding the vertices in $W$, and arcs in $B\cup\{(w_i,x):x\in T_i,i\in\{1,\ldots,p\}\}$. In addition, we obtain our desired $D_{p,r}$ from $D_p$ by adding vertices in $Z$ and arcs in $B'\cup\{(z_i,y):y\in Z_i,i\in\{1,\ldots,r\} \}$. As a consequence, we observe that $F$ can be constructed in such a described process, \emph{i.e.}, $F=D_{p,r}\in \mathcal{D}_D$, and therefore in $F\in \mathcal{D}_0$. 
\end{proof}

Let us mention, without a proof, that if arcs of $K_{1,t}$ are oriented such that $t_1>0$ leaves are sources and $t_2=t-t_1$ are sinks, then we can proceed as follows to describe the structure of a digraph $F$ such that $F\Box K_{1,t}$ is an ECD digraph. The vertices of $F$ must be partitioned into $t_1+3$ sets $W_1,\ldots,W_{t_1},W_{t_1+1},W_{t_1+2},W_{t_1+3}$ such that every set $W_i$, with $i\in\{1,\ldots,t_1\}$, is an ECD set of $F$. Further, the set $W_{t_1+1}$ must be an ECD set for the digraph induced by the vertices $W_{t_1+1}\cup W_{t_1+2}\cup W_{t_1+3}$ and all out-arcs from $W_{t_1+1}$ end in vertices of $W_{t_1+2}\cup W_{t_1+3}$. On the other hand, $W_{t_1+2}$ must be an ECD set in a digraph induced by $W_{t_1+2}\cup W_{t_1+3}\cup \cup_{i=1}^{t_1}W_{i}$ with no arcs starting in $W_{t_1+2}$ and ending in $W_{t_1+1}$. Notice that in this construction, the arcs that starts in $W_{t_0+3}$ and ends everywhere else are also possible. With this construction it is not hard to see that the set
$$S=\bigcup_{i=1}^{t_1}(W_i\times \{u_i\})\cup (W_{t_1+1}\times\{v\})\cup (W_{t_1+2}\times\{u_{t_1+1},\ldots,u_t\})$$
forms an ECD set of $F\Box K_{1,t}$. The converse, if the product $F\Box K_{1,t}$ is an ECD Cartesian product digraph, then $F$ must be as described, holds as well.


\section{Direct product}

The main goal of this section concerns describing all ECD digraphs among direct product of digraphs with $\delta(v)\leq 2$, this means among paths and cycles. The reason for this is hidden in (\ref{dirdeg}) because we are limited with the vertices of the product. Another motivation is that ECD graphs among the direct product of graphs was a hard task and was settled completely in three papers, see \cite{JeKlSp,KlSpZe,Zero}. We will solve this more elegant for directed cycles due to an old result about the structure of the direct product of digraphs. We recall only a special version of this results for cycles.

\begin{theorem}[\cite{Mcan}, Theorem 2]
\label{dirprodstructure} The direct product $C^0_{k_1}\times \cdots \times C^0_{k_t}$ has exactly $\frac{k_1\cdots k_t}{lcm(k_1,\ldots,k_t)}=gcd(k_1,\ldots,k_t)$ connected components.
\end{theorem}

Since there are no sinks (and no sources) in the cycles of the theorem above, by applying (\ref{dirdeg}), we can note that every vertex $x$ of $C=C^0_{k_1}\times \cdots \times C^0_{k_t}$ has $\delta^-_C(x)=1=\delta^+_C(x)$. Therefore, every component of $C$ must be a cycle with no sinks (and no sources) and we obtain the following corollary.

\begin{corollary}
\label{cyclestructure} The direct product $C^0_{k_1}\times \cdots \times C^0_{k_t}$ is isomorphic to $gcd(k_1,\ldots,k_t)C^0_{lcm(k_1,\ldots,k_t)}$.
\end{corollary}

As a cycle $C^0_k$ is an ECD graphs if and only if $k$ is an even number, we immediately obtain the result for direct product of cycles without sinks (and without sources).

\begin{corollary}
\label{cycles} The direct product $C^0_{k_1}\times \cdots \times C^0_{k_t}$ is an ECD graph if and only if at least one number from $\{k_1,\ldots,k_t\}$ is even.
\end{corollary}

However, our ambition is higher and we would like to characterize all ECD digraphs among all directed cycles with sinks and sources. Notice that in this case Theorem \ref{dirprodstructure} does not hold anymore. A source $v$ of a cycle $C$, or of a path $P$, is a \emph{neighboring source} of a vertex $u$, if there exists a directed path from $v$ to $u$, that is, this path contains no sinks with the possible exception of $u$. Similarly, a sink $w$ is a \emph{neighboring sink} of a vertex $u$, if there exists a directed path from $u$ to $w$. The distance between a vertex and its neighboring sink or source is the length of the path between them that contains no other sources or sinks.

\begin{theorem}\label{ecddircycle}
The direct product $C_{k_1}^{p_1}\times\cdots\times C_{k_t}^{p_t}$ is an ECD digraph if and only if $p_1=\cdots=p_{t-1}=0$ and
\begin{itemize}
\item[{\rm (i)}] if $p_t=0$, then at least one number of $\{k_1,\ldots,k_t\}$ is even, or

\item[{\rm (ii)}] if $p_t>0$, then every sink is at even distance to at least one of its neighboring sources.
\end{itemize}
\end{theorem}
	
\begin{proof}
Let $D=C_{k_1}^{p_1}\times\cdots\times C_{k_t}^{p_t}$ and $D'=C_{k_1}^{p_1}\times\cdots\times C_{k_{t-1}}^{p_{t-1}}$. If $p_1=\cdots=p_{t}=0$, then $D\cong gcd(k_1,\ldots,k_t)C^0_{lcm(k_1,\ldots,k_t)}$ by Corollary \ref{cyclestructure}. If one of $\{k_1,\ldots,k_t\}$ is even, then $D$ is an ECD digraph by Corollary \ref{cycles}. So, we may assume $p_1=\cdots=p_{t-1}=0$, $p_t>0$, and every sink is at even distance to at least one of its neighboring sources in $C_{k_t}^{p_t}$. Again, $D'\cong gcd(k_1,\ldots,k_{t-1})C^0_{lcm(k_1,\ldots,k_{t-1})}$ by Corollary \ref{cyclestructure}. Therefore, we may observe every component of $D'$ separately, which means that we deal with $C=C^0_{lcm(k_1,\ldots,k_{t-1})}\times C_{k_t}^{p_t}$. Let $R=\{r_1,\ldots,r_{p_t}\}$ be all sources and $Q=\{q_1,\ldots,q_{p_t}\}$ be all sinks of $C_{k_t}^{p_t}$. Let $A$ be the subset of all vertices $z$ from $V(C_{k_t}^{p_t})$ such that the path between $z$ and its neighboring source is of even length. In addition, let $Q'\subseteq Q$ be the set containing all sinks for which both paths to his neighboring sources are of even distance. We will show that $S=V(C^0_{lcm(k_1,\ldots,k_{t-1})})\times S'$ is an ECD set of $C$ for $S'=A\cup R\cup Q'$.

First notice that $S'$ is an independent set of $C_{k_t}^{p_t}$. This yields also the independence of $S$. On the other hand, the set $S''=V(C_{k_t}^{p_t})-S'$ may induce some isolated arcs, but every such arc ends in a sink, that had exactly one of its neighboring sources at even distance. Moreover, every vertex from $S''$ has exactly one in-neighbor in $S'$. In particular, a sink that have one source at odd distance and the other source at even distance, has its in-neighbor on the odd side in $S'$. Therefore, every vertex from $V(C^0_{lcm(k_1,\ldots,k_{t-1})})\times S''$ has exactly one in-neighbor in $S$ and $S$ is an ECD set of $C$.

Conversely, let $C_{k_1}^{p_1}\times C_{k_t}^{p_t}$ be an ECD digraph with an ECD set $S$. With a purpose of contradiction, suppose that for instance $p_{t-1},p_t>0$. Without loss of generality, we chose the notation so that $v_2$ is a sink of $C_{k_{t-1}}^{p_{t-1}}$ and $u_2$ is a source of $C_{t_k}^{p_k}$. By (\ref{dirdeg}) every vertex of $D$ that projects to a sink (source) in one factor is a sink (source) in $D$ as well. Now, notice that the vertices $\{(v_1,u_2).(v_2,u_1),(v_2,u_3),(v_3,u_2)\}$ induce a cycle $C_4^2$ in $C_{k_{t-1}}^{p_{t-1}}\times C_{k_{t}}^{p_{t}}$ with sinks $(v_2,u_1)$ and $(v_2,u_3)$ and sources $(v_1,u_2)$ and $(v_3,u_2)$. Also, for any vertex $x$ from $C_{k_1}^{p_1}\times \cdots \times C_{k_{t-2}}^{p_{t-2}}$, the vertices $(x,v_2,u_1)$ and $(x,v_2,u_3)$ are sinks and the vertices $(x,v_1,u_2)$ and $(x,v_3,u_2)$ are sources from $D$. Since every source of an ECD digraph must be in any ECD set, it must happen $(x,v_2,u_1),(x,v_2,u_3)\in D$, which is a contradiction because $(x,v_1,u_2)$ and $(x,v_3,u_2)$ have both $(x,v_2,u_1)$ and $(x,v_2,u_3)$ as in-neighbors. Therefore, at most one number from $p_1,\ldots,p_t$ is different from 0 and we may assume by commutativity of the direct product that $p_1=\cdots =p_{t-1}=0$. If in addition also $p_t=0$, then (i) follows by Corollary \ref{cycles}.

Thus we may assume that $p_t>0$. We have $C^0_{k_1}\times \cdots \times C^0_{k_{t-1}}\cong gcd(k_1,\ldots,k_t)C^0_{lcm(k_1,\ldots,k_t)}$ by Corollary \ref{cyclestructure}. Therefore, it is enough to study only one component $D'=C^0_k\times C_{k_t}^{p_t}$ for $k=lcm(k_1,\ldots,k_t)$. In order to produce a contradiction, suppose that there exists a sink $q$ that has odd distance to both of its neighboring sources, say $r$ and $r'$. Let $rr_1r_2\ldots r_{k_1}q$ and $r'r'_1r'_2\ldots r'_{k_2}q$ where $k_1$ and $k_2$ are even numbers be the paths between $q$ and $r$ or $r'$. Clearly,  the vertex $(v_k,q)$ is a sink with $\delta^-_{D'}(v_1,q)=2$ by (\ref{dirdeg}). Also, there are exactly two oriented paths in $D'$ that end in $(v_1,q)$. These paths are
$$(v_{k-k_1-1},r)(v_{k-k_1},r_1)(v_{k-k_1+1},r_2)\ldots(v_{k-2},r_{k_1-1})(v_{k-1},r_{k_1})(v_k,q)$$
and
$$(v_{k-k_2-1},r')(v_{k-k_2},r'_1)(v_{k-k_2+1},r'_2)\ldots(v_{k-2},r'_{k_2-1})(v_{k-1},r'_{k_2})(v_k,q).$$
Moreover, the in-degree and out-degree of every inner vertex of these two paths is one. Because every source is in every ECD set of an ECD digraph, we have that $(v_{k-k_1-1},r),(v_{k-k_2-1},r')\in S$. This implies that vertices
$$(v_{k-k_1+1},r_2),(v_{k-k_2+1},r'_2),(v_{k-k_1+3},r_4),(v_{k-k_2+3},r'_4),\ldots (v_{k-1},r_{k_1}),(v_{k-1},r'_{k_2})$$
belong to $S$ as well so that above mentioned paths are efficiently dominated. But then $(v_k,q)$ is adjacent from two vertices $(v_{k-1},r_{k_1})$ and $(v_{k-1},r'_{k_2})$ of $S$, which is a contradiction. Thus, every sink must have at least one of its neighboring sources at even distance, (ii) follows and the proof is complete. 
\end{proof}

We next continue with the case of paths. Here we cannot use Theorem \ref{dirprodstructure}, because paths are not strongly connected, see \cite{Mcan}.

\begin{theorem}
\label{ecddirpath} The direct product $P_{k_1}\times\cdots\times P_{k_t}$ is an ECD digraph if and only if $P_{k_i}$ does not contain a sink of degree 2 for every $i\in\{1,\ldots,t\}$.
\end{theorem}
	
\begin{proof}
Suppose first that one path, say $P_{k_1}$, has a sink $s$ of degree 2 and let $x$ and $y$ be the in-neighbors of $s$ on $P_{k_1}$. Suppose that one other path, say $P_{k_2}$, has a source $q$ of out-degree 1 and that $w$ is adjacent from $q$. By (\ref{dirdeg}), the vertex $(s,w)$ is a sink of $P_{k_1}\times P_{k_2}$ with exactly two in-neighbors $(x,q)$ and $(y,q)$. Moreover, these three vertices form a component $C$ of $P_{k_1}\times P_{k_2}$. By (\ref{dirdeg}), the vertices $(x,q)$ and $(y,q)$ are sources of out degree 1. Since all sources must be in every ECD set, it is readily seen that $P_{k_1}\times P_{k_2}$ is not an ECD digraph, as two sources have the same out-neighbor. We may assume now that both vertices of degree one in $P_{k_2}$ are sinks which implies that there exists a source $q$ in $P_{k_2}$ of out-degree 2. Let $u$ and $v$ be out-neighbors from $q$. Note that the vertices $(s,u)$ and $(s,v)$ are sinks adjacent from vertices $(x,q)$ and $(y,q)$ which are sources. Moreover, these four vertices form a component $C'$ of $P_{k_1}\times P_{k_2}$. Again we see that $P_{k_1}\times P_{k_2}$ is not an ECD digraph. Thus, we have proved the statement for $t=2$.

Let now $t>2$ and assume by induction that the direct product of $t-1$ paths is not an ECD digraph if there exists a sink of degree 2 in one of these paths.  If $P_{k_1}$ has a sink of degree 2 and $P_{k_2}$ a source of degree 1, then by the same notation as above, there exists a component $C$ of $P_{k_1}\times P_{k_2}$ that is a path with a sink of degree 2. Hence, the product $C\times P_{k_3}\times\cdots\times P_{k_t}$ is not an ECD digraph by induction hypothesis, and therefore, also $P_{k_1}\times\cdots\times P_{k_t}$ is not an ECD digraph. Now we may assume that all vertices of degree one are sinks in all paths $P_{k_2},\ldots,P_{k_t}$. We consider a component $C'$ in $P_{k_1}\times P_{k_2}$. It is easy to check that the direct product of $C'$ and an arbitrary arc results in a subdigraph isomorphic to $C'$. Even more, if a neighbor $x$ of a sink $s$ of degree one is not a source, then this subdigraph is a component isomorphic to $C'$. If $x$ is a source, then all vertices of a layer through $x$ are sources and must therefore be in every ECD set if the digraph is an ECD digraph. The product $P_{k_1}\times\cdots\times P_{k_t}$ is therefore not an ECD digraph because it contains a component isomorphic to $C'$ or $C'\times(x,s)\cong C'$ where $x$ is a source.

Conversely, suppose that $P_{k_i}$ does not contain a sink of degree 2 for every $i\in\{1,\ldots,t\}$. This means that any path $P_{k_i}$ has either one source of degree 2 and two sinks of degree one, or one source and one sink both of degree 1, for every $i\in\{1,\ldots,t\}$. By (\ref{dirdeg}), every vertex of degree more than 2 in $P_{k_1}\times\cdots \times P_{k_t}$ is a source. Also, a vertex of degree 2 can be a source, but in this case this vertex is the only source in a component which is a path. All other vertices of degree 2 have exactly one in-neighbor and one out-neighbor. All sinks are in vertices of degree 1, again by (\ref{dirdeg}). Thus, every component of the underlying graph of $P_{k_1}\times\cdots\times P_{k_t}$ is without cycles. A set that contains all sources of $P_{k_1}\times\cdots\times P_{k_t}$ together with all vertices at even distance from these sources form an ECD set of $P_{k_1}\times\cdots\times P_{k_t}$ and this digraph is an ECD digraph. 
\end{proof}


\section{Strong and lexicographic products}

The most ``friendly'' product with respect to ECD digraphs is the strong product due to (\ref{strongneig}). The next result is therefore not highly surprising.

\begin{theorem}
\label{ecdstrong} Let $D$ and $F$ be digraphs. The strong product $D\boxtimes F$
is an ECD digraph if and only if $D$ and $F$ are ECD graphs.
\end{theorem}

\begin{proof}
Let first $D$ and $F$ be ECD digraphs with ECD sets $S_D$ and $S_F$, respectively. We claim that $S_D\times S_F$ is an ECD set of $D\boxtimes F$. By (\ref{strongneig})
$$V(D\boxtimes F)\subseteq \bigcup_{(d,f)\in S_D\times S_F}N_{D\boxtimes F}^+[(d,f)].$$
Suppose that there exist different vertices $(d,f)$ and $(d'f')$ from $S_D\times S_F$ such that the intersection of their closed out-neighborhoods is nonempty. This means that there exists $(d_0,f_0)\in N_{D\boxtimes F}^+[(d,f)]\cap N_{D\boxtimes F}^+[(d',f')]$. If $d=d'$, then $f\neq f'$, and by (\ref{strongneig}) we have
$$N_{D\boxtimes F}^+[(d,f)]\cap N_{D\boxtimes F}^+[(d',f')]=\left(N_{D}^+[d]\times N_{F}^+[f]\right)\cap \left(N_{D}^+[d]\times N_{F}^+[f']\right)=N_{D}^+[d]\times \left(N_{F}^+[f]\cap N_{F}^+[f']\right).$$
Thus, $f_0\in N_{F}^+[f]\cap N_{F}^+[f']$ which is not possible for an ECD set $S_F$. We obtain a contradiction by symmetric arguments if $f=f'$. So, suppose that $d\neq d'$ and $f\neq f'$. Again by (\ref{strongneig}) we obtain that  $d_0\in N_{D}^+[d]\cap N_{D}^+[d']$ and $f_0\in N_{F}^+[f]\cap N_{F}^+[f']$, a contradiction with $S_D$ and $S_F$ being ECD sets of $D$ and $F$, respectively. Therefore,
$$N_{D\boxtimes F}^+[(d,f)]\cap N_{D\boxtimes F}^+[(d',f')]=\emptyset$$
for every two different vertices of $S_D\times S_F$ and $D\boxtimes F$ is an ECD digraph.

Conversely, let $D\boxtimes F$ be an ECD digraph with an ECD set $S$. Let $f$ be an arbitrary vertex from $F$. Every vertex from $D^f$ is in exactly one closed neighborhood of a vertex from $S$. Denote by $S_f$ the set of all vertices from $S$ whose closed neighborhoods have a nonempty intersection with $D^f$. We will show that $p_D(S_f)$ form an ECD set of $D$. Let first $d$ and $d'$ be two different vertices from $p_D(S_f)$ and suppose that $(d,f_1)$ and $(d',f_2)$ are from $S_f$. If the intersection of closed out-neighborhoods centered in $d$ and in $d'$ is nonempty, that is $d_0\in N_D^+[d]\cap N_D^+[d']$, then $(d_0,f)\in N_{D\boxtimes F}^+[(d,f_1)]\cap N_{D\boxtimes F}^+[(d',f_2)]$, which is a contradiction, because $(d,f_1)$ and $(d',f_2)$ are different and belong to an ECD set of $D\boxtimes F$. On the other hand, $V(D)\subseteq \bigcup_{d\in p_D(S_f)}N_D^+[d]$, because $S$ is an ECD set of $D\boxtimes F$, and by (\ref{strongneig}). Therefore, is $p_D(S_f)$ an ECD set of $D$ and $D$ is an ECD digraph. By using symmetric arguments, we can show that also $F$ is an ECD digraph and the proof is completed. 
\end{proof}

By observing (\ref{lexneig}), we can not expect that ECD digraphs among the lexicographic products will be as rich as in the case of other standard products. Nevertheless, there are two main exceptions, from which at least one can be considered as trivial, see (i) of the following theorem.

\begin{theorem}
\label{ecdlex}Let $D$ and $F$ be digraphs. The lexicographic product $D\circ F$
is an ECD digraph if and only if

\begin{itemize}
\item[{\rm (i)}] $D$ is a digraph without arcs and $F$ is an ECD digraph, or

\item[{\rm (ii)}] $D$ is an ECD digraph and $F$ contains an out-universal vertex.
\end{itemize}
\end{theorem}

\begin{proof}
If $D$ is a digraph on $n$ vertices without edges,
then $D\circ F$ is isomorphic to $n$ copies of $F$. If in addition $F$ is an ECD digraph, then also $n$ copies of $F$ form an ECD digraph. Now, let $D$ be an ECD digraph, let $S_D$ be
an ECD set and let $f_{0}$ be an out-universal vertex of $F$. We
will show that $S_D\times \{f_{0}\}$ is an ECD set
of $G\circ H$. Because $f_0$ is an out-universal vertex of $F$ we have that $N^-_{D\circ F}[(d,f_{0})]=N_{D}[(d)]\times
V(F)$ by (\ref{lexneig}), and so, $\bigcup _{d\in S_D}N_{D\circ F}[(d,f_{0})]=V(G\times H)$.
If $d,d'\in S_D$ and $d\neq d'$, then $N_{D\circ F}[(d,f_{0})]\cap N_{D\circ F}[(d',f_{0})]\neq \emptyset $ implies that $N_{D}[d]\cap N_{D}[d']\neq \emptyset $, which is a contradiction. Therefore, $D\circ F$ is an ECD graph.

Conversely, let $D\circ F$ be an ECD graph with an ECD set $S$. Let $(d,f)$ be an arbitrary vertex from  $S$.
Suppose first that $f$ is not an out-universal vertex in $F$. Then there exists a vertex, say $(d,f')$, in $^{d}\!F-N^+_{D\circ F}[(d,f)]$. Because $S$ is an ECD set, there exists $(d_0,f_0)$ such that $(d,f')\in N_{D\circ F}^+[(d_0,f_0)]$. If $d_0\neq d$, then $(d,f)\in N_{D\circ F}^+[(d_0,f_0)]\cap N_{D\circ F}^+[(d,f)]$, which is not possible for two different vertices $(d_0,f_0)$ and $(d,f)$ from $S$. Therefore, we have $d_0=d$. If $\delta_D(d)^+>0$, then again $(d',f)\in N_{D\circ F}^+[(d_0,f_0)]\cap N_{D\circ F}^+[(d,f)]$ for any $d'\in N_D^+(d)$, which is not possible. So, $\delta_D(d)^+=0$ for every vertex $d\in V(D)$ such that $(d,f)\in S$. Since every vertex $(d'',f'')\in V(D\circ F)$ is in exactly one $N_{D\circ F}^+[(d,f)]$, with $(d,f)\in S$, we infer that $d''\neq d$ yields $\delta_D(d)^+>0$, which is not possible. Therefore, $d''=d$ and $\delta_D(d'')^+=0$ for any $d''\in V(D)$. But then $D$ has no arcs at all and it is an empty digraph. To see that $F$ is an ECD digraph, we observe any $F$-layer $^{d}\!F$ (which is always isomorphic to $F$). Clearly, $^{d}\!F\cap S$ form an ECD set of $^{d}\!F\cong F$ and $F$ is an ECD digraph and (i) follows.

Now we may assume that $f$ is an out-universal vertex (recall that $(d,f)$ is an arbitrary vertex from $S$). Notice that in this case $^{d}\!F\subseteq N_{D\circ F}^+[(d,f)]$. We claim that $p_D(S)$ is an ECD set of $D$. If not, then there exist $d,d'\in p_D(S)$ such that $N_D^+(d)\cap N_D^+(d')\neq \emptyset$. This fact together with (\ref{lexneig}) yield a contradiction with $S$ being an ECD set of $D\circ H$. Therefore, $p_D(S)$ is an ECD set of $D$ and (ii) follows. 
\end{proof}

\end{document}